\documentclass[10pt]{amsart}
\usepackage{amsfonts,amssymb,amscd,amsmath,enumerate,verbatim,calc,amsthm}
\input xy
\xyoption{all}

\headsep=1cm \numberwithin{equation}{section}
\newtheorem{thm}{Theorem}[section]
\newtheorem{cor}[thm]{Corollary}
\newtheorem{lem}[thm]{Lemma}
\newtheorem{prop}[thm]{Proposition}
\newtheorem{defn}[thm]{Definition}

\newtheorem{rem}[thm]{Remark}
\newtheorem{ques}[thm]{Question}

\newcommand{\Ann}{\mbox{Ann}\,}
\newcommand{\width}{\mbox{width}\,}
\newcommand{\vdim}{\mbox{vdim}\,}
\newcommand{\coker}{\mbox{Coker}\,}
\newcommand{\Hom}{\mbox{Hom}\,}
\newcommand{\Ext}{\mbox{Ext}\,}
\newcommand{\Tor}{\mbox{Tor}\,}
\newcommand{\Spec}{\mbox{Spec}\,}
\newcommand{\Max}{\mbox{Max}\,}

\newcommand{\Ass}{\mbox{Ass}\,}
\newcommand{\ass}{\mbox{ass}\,}

\newcommand{\Supp}{\mbox{Supp}\,}

\newcommand{\depth}{\mbox{depth}\,}
\renewcommand{\dim}{\mbox{dim}\,}

\newcommand{\grd}{\mbox{grade}\,}

\newcommand{\pd}{\mbox{pd}\,}
\newcommand{\id}{\mbox{id}\,}
\newcommand{\fd}{\mbox{fd}\,}

\newcommand{\h}{\mbox{ht}\,}

\newcommand{\E}{\mbox{E}}

\renewcommand{\H}{\mbox{H}}

\newcommand{\Soc}{\mbox{Soc}\,}

\newcommand{\fa}{\mathfrak{a}}
\newcommand{\fb}{\mathfrak{b}}

\newcommand{\fm}{\mathfrak{m}}
\newcommand{\fp}{\mathfrak{p}}

\begin{document}
\bibliographystyle{amsplain}


\title[Characterizing local rings]
 {Characterizing local rings  via perfect and coperfect modules}

\bibliographystyle{amsplain}

     \author[M. Rahmani]{Mohammad Rahmani}
     \author[A.- J. Taherizadeh]{Abdoljavad Taherizadeh}

\address{Faculty of Mathematical Sciences and Computer,
Kharazmi University, Tehran, Iran.}

\email{m.rahmani.math@gmail.com}
\email{taheri@khu.ac.ir}



\begin{abstract}
Let $R$ be a Noetherian ring and let $C$ be a semidualizing $R$-module. In this paper, by using the classes $ \mathcal{P}_C $ and $ \mathcal{I}_C $, we extend the notions of perfect and coperfect modules introduced by D.Rees \cite{R} and O.Jenda \cite{J1}. First, we study the basic properties of these modules and relations between them. Next, we characterize local rings in terms of the existence of special perfect (resp. coperfect) modules.
\end{abstract}

\maketitle

\bibliographystyle{amsplain}
\section{introduction}

Throughout this paper, $R$ is a commutative Noetherian ring with non-zero identity. Recall that the grade of $ M $,
defined by D. Rees \cite{R}, is the least integer $ i \geq 0 $ such that $ \Ext^i_R(M,R) \neq 0 $. He showed that if $M$ is finitely generated, then $ \grd_R(M) $ is the length of a maximal $R$-sequence contained in $ \Ann_R(M) $. An $R$-module $M$ is called perfect (resp. $ n $-perfect) if one has $ \grd_R(M) = \pd_R(M) $ (resp. $ \grd_R(M) = n = \pd_R(M) $) where $n$ is a non-negative integer. He, also, showed that $ \grd_R(M) $ is the least integer $ i \geq 0 $ such that $ \Ext^i_R(M,P) \neq 0 $ for some projective $R$-module $ P $. As a dual notion of grade, O. Jenda \cite{J1} introduced an invariant, namely $ E $-dimension. For an $R$-module $M$, $ E $-dimension of $ M $, denoted by $ E\emph{-}\dim_R(M) $, is the least integer $ i \geq 0 $ such that $ \Ext^i_R(I,M) \neq 0 $ for some injective $R$-module $ I $. Also, $M$ is called a coperfect (resp. $ n $-coperfect) $R$-module if one has $ E\emph{-}\dim_R(M) = \id_R(M) $ (resp. $ E\emph{-}\dim_R(M) = n = \id_R(M) $). A finitely generated $R$-module $C$ is semidualizing if the natural homothety map $ R\longrightarrow \Hom_R(C,C) $ is an isomorphism and $ \Ext^i_R(C,C)=0 $ for all $ i>0 $. Semidualizing modules have been studied by Foxby \cite{F}, Vasconcelos \cite{V}
and Golod \cite{G}. For a semidualizing $R$-module $C$, the class of $ C $-projectives (resp. $ C $-injectives), denoted $ \mathcal{P}_C $ (resp. $ \mathcal{I}_C $), consists of those $ R $-modules of the form $ C \otimes_R P $ (resp. $ \Hom_R(C,I) $) for some projective (resp. injective) $ R $-module $ P $ (resp. $ I $). In \cite{HW}, H. Holm and D. White showed that for an $R$-module $M$, there exists a projective resolution (resp. injective coresolution) by modules from $ \mathcal{P}_C $ (resp. $ \mathcal{I}_C $). Despite the fact that these (co)resolutions may not be exact, they still have good lifting properties. By using these (co)resolutions, R. Takahashi and D. White \cite{TW} defined two new homological dimensions, namely $ C\emph{-}\pd $ and $ C\emph{-}\id $. Recently, B. Kubik \cite{K}, introduced the dual notion of semidualizing modules, namely quasidualizing modules. Over a local ring $ (R, \fm) $, an Artinian $R$-module $T$ is quasidualizing if the natural homothety map $ \widehat{R} \longrightarrow \Hom_R(T,T) $ is an isomorphism and $ \Ext^i_R(T,T)=0 $ for all $ i>0 $. In section 2, for an $R$-module $M$, we introduce two invariant for modules,
 namely $ \mathcal{P}_C\emph{-}\grd_R(M) $ and $ \mathcal{I}_C\emph{-}\grd_R(M) $ as generalization of the classical invariants $ \grd_R(M) $ and $ E\emph{-}\dim_R(M) $, respectively. We study the basic properties of these new invariants and the relations of them with the relative homological dimensions $ C\emph{-}\pd_R(M) $ and $ C\emph{-}\id_R(M) $. For instance, the following result, proved in Theorem 3.11, is a generalization of \cite[Theorem 3.1]{J1}.\\
 \textbf{Theorem A.} \emph{Let $M$ be a finitely generated $R$-module and let $I$ be an injective $R$-module for which $ \emph{\Hom}_R(M,I) \neq 0 $. Then }\\
 \centerline{$ \mathcal{P}_C \emph{-}\grd_R(M) \leq \mathcal{I}_C \emph{-}\grd_R(\Hom_R(M,I)) \leq C \emph{-}\pd _R(M) $.}

Next, we define the notions of
 $ C $-perfect (resp. $ n $-$ C $-perfect)  and $ C $-coperfect (resp. $ n $-$ C $-coperfect) modules, and use them to impose some conditions on
 $ C $ to be dualizing. The next result is Theorem 3.19.  \\
  \textbf{Theorem B.} \emph{Let $ (R, \fm) $ be a local ring with $ \emph{\dim}(R) = n $ an let $C$ be a semidualizing $R$-module. The following are equivalent:
  \begin{itemize}
  	\item[(i)]{$C$ is dualizing.}
  	\item[(ii)]{$E(R/ \fm) $ is  $ n $-$ C $-perfect.}
  	\item[(iii)]{$\widehat{R}$ is $ n $-$ C $-coperfect.}
  \end{itemize}}
 We, also, use $ n $-perfect (resp. $ n $-coperfect) modules to characterize local rings. We find some special $ n $-perfect (resp. $ n $-coperfect) modules whose existence imply the Regularity (resp. Gorensteinness) of the ring. For instance, the following result is Theorem 3.20. \\
   \textbf{Theorem C.} 	\emph{Let $ (R, \fm) $ be a local ring with $ \emph{\dim}(R) = n $. The following are equivalent:
   \begin{itemize}
   	\item[(i)]{$R$ is Gorenstein.}
   	\item[(ii)]{any quasidualizing $R$-module is $ n $-perfect.}
   	\item[(iii)]{there exists an $ n $-perfect quasidualizing $R$-module with $ 1 $-dimensional socle.}
   \end{itemize}}
\section{preliminaries}

In this section, we recall some definitions and facts which are needed throughout this
paper. By an injective cogenerator, we always mean an injective $R$-module $E$ for which $ \Hom_R(M,E) \neq 0 $ whenever $M$ is a nonzero $R$-module. For an $R$-module $M$, the injective hull of $M$, is always denoted by $E(M)$.

\begin{defn}
	\emph{Let $\mathcal{X} $ be a class of $R$-modules and $M$ an $R$-module. An $\mathcal{X}$-\textit{resolution} of $M$ is a complex of $R$-modules in $\mathcal{X} $ of the form \\
		\centerline{ $X = \ldots \longrightarrow X_n \overset{\partial_n^X} \longrightarrow X_{n-1} \longrightarrow \ldots \longrightarrow X_1 \overset{\partial_1^X}\longrightarrow X_0 \longrightarrow 0$}
		such that $\H_0(X) \cong M$ and $\H_n(X) = 0$ for all $ n \geq 1$.}
	\emph{Also the $ \mathcal{X}$-\textit{projective dimension} of $M$ is the quantity \\
		\centerline{ $ \mathcal{X}$-$\pd_R(M) := \inf \{ \sup \{ n \geq 0 | X_n \neq 0 \} \mid X$ is an $\mathcal{X}$-resolution of $M \}$}.}
	\emph{So that in particular $\mathcal{X}$-$\pd_R(0)= - \infty $. The modules of $\mathcal{X}$-projective dimension zero are precisely the non-zero modules in $\mathcal{X}$. The terms of $\mathcal{X}$-\textit{coresolution} and $\mathcal{X}$-$\id$ are defined dually.}
\end{defn}

\begin{defn}
 \emph{A finitely generated $ R $-module $ C $ is \textit{semidualizing} if it satisfies the following conditions:
\begin{itemize}
             \item[(i)]{The natural homothety map $ R\longrightarrow \Hom_R(C,C) $ is an isomorphism.}
             \item[(ii)]{$ \Ext^i_R(C,C)=0 $ for all $ i>0 $.}
          \end{itemize}
         For example a finitely generated projective $R$-module of rank 1 is semidualizing. An $R$-module $D$ is \textit{dualizing} if it is semidualizing and that $\id_R (D) < \infty $. Moreover, $D$ is \textit{pointwise dualizing} if $ D_{\fm} $ is dualizing $R_{\fm}$-module for any $ \fm \in \Max(R) $. For example the canonical module of a Cohen-Macaulay local ring, if exists, is dualizing.}

\emph{Assume that $ (R, \fm) $ is local. Following \cite{K}, an Artinian $ R $-module $ T $ is called \textit{quasidualizing} if it satisfies the following conditions:
\begin{itemize}
	\item[(i)]{The natural homothety map $ \widehat{R} \longrightarrow \Hom_R(T,T) $ is an isomorphism.}
	\item[(ii)]{$ \Ext^i_R(T,T)=0 $ for all $ i>0 $.}
\end{itemize}}
\end{defn}
For example $ E(R/\fm) $ is a quasidualizing $ R $-module.

\begin{defn}
\emph{Following \cite{HJ}, let $C$ be a semidualizing $R$-module. We set
\begin{itemize}
	\item[]{$\mathcal{P}_C(R) =$ the subcategory of $R$--modules  $C \otimes_R P$ where $P$ is a projective $R$-module.}
	\item[]{ $\mathcal{I}_C(R) =$ the subcategory of $R$--modules  $\Hom_R(C,I) $ where $I$ is an injective $R$-module.}
\end{itemize}
The $R$-modules in $\mathcal{P}_C(R)$ and $\mathcal{I}_C(R)$ are called $C$-\textit{projective} and $C$-\textit{injective}, respectively.  If $ C = R $, then it recovers the classes of projective and injective modules, respectively.
We use the notations $C$-$\pd$ and $C$-$\id$ instead of $\mathcal{P}_C$-$\pd$ and $\mathcal{I}_C$-$\id$, respectively.}
\end{defn}

\begin{prop}\label{B0}
Let $C$ be a semidualizing $R$-module. Then we have the following:
\begin{itemize}
           \item[(i)]{$\emph\Supp (C) = \emph\Spec (R)$, $\emph\dim(C) = \emph\dim(R)$ and $\emph\Ass(C) = \emph\Ass(R)$.}
           \item[(ii)] { If $R \rightarrow S$ is a flat ring homomorphism, then $ C \otimes_R S$ is a semidualizing $S$-module.}
            \item[(iii)]{ If $ x \in R $ is $R$--regular, then $C/ xC$ is a semidualizing $R/ xR$-module.}
             \item[(iv)]{$\emph\depth_R (C) = \emph\depth (R) $.}
             \end{itemize}
\end{prop}
\begin{proof}
 The parts (i), (ii) and (iii) follow from the definition of semidualizing modules. For (iv), note that an element of $R$ is $R$-regular if and only if it is $C$-regular since $\Ass (C) = \Ass (R)$. Now an easy induction yields the equality.
\end{proof}

\begin{defn}
	\emph{Let $C$ be a semidualizing $R$-module. The \textit{Auslander class with
			respect to} $C$ is the class $\mathcal{A}_C(R)$ of $R$-modules $M$ such that:
		\begin{itemize}
			\item[(i)]{$\Tor_i^R(C,M) = 0 = \Ext^i_R(C, C \otimes_R M)$ for all $i \geq 1$, and}
			\item[(ii)]{The natural map $ M \rightarrow \Hom_R(C , C \otimes_R M )$ is an isomorphism.}
		\end{itemize}		
		The \textit{Bass class with
			respect to} $C$ is the class $\mathcal{B}_C(R)$ of $R$-modules $M$ such that:
		\begin{itemize}
			\item[(i)]{$\Ext^i_R(C,M) = 0 = \Tor_i^R(C, \Hom_R(C,M))$ for all $i \geq 1$, and}
			\item[(ii)]{The natural map $ C \otimes_R \Hom_R(C,M)) \rightarrow M $ is an isomorphism.}
		\end{itemize}
		The class $\mathcal{A}_C(R)$ contains all $R$-modules of finite projective dimension and those of finite $C$-injective dimension. Also the class $\mathcal{B}_C(R)$ contains all $R$-modules of finite injective dimension and those of finite $C$-projective dimension (see \cite[Corollary 2.9]{TW}). Also, if any two $ R $-modules in a short exact sequence are in $ \mathcal{A}_C(R) $ (resp. $ \mathcal{B}_C(R) $), then so is the third (see \cite[Corollary 6.3]{HW}).}
\end{defn}

\begin{defn}
\emph{Let $M$ be a non-zero $R$-module. Following \cite{R}, the \textit{grade} of $M$, denoted by $ \grd_R(M) $, is defined to be \\
	\centerline{$ \grd_R(M) = \inf\{ i \geq 0 \mid \Ext_R^i(M,R) \neq 0 \}$.}
	Also, $M$ is called a \textit{perfect} $R$-module precisely when $ \grd_R(M) = \pd_R(M) $.
	Following \cite{J1}, the $ E $-\textit{dimension} of $M$, denoted by $ E\emph{-}\dim_R(M) $, is defined to be \\
\centerline{$ E\emph{-}\dim_R(M) = \inf\{ i \geq 0 \mid \Ext_R^i(E,M) \neq 0 $ for some injective $R$-module $E$ $\}$.}
Also, $M$ is called a \textit{coperfect} $R$-module precisely when $ E\emph{-}\dim_R(M) = \id_R(M) $.}
\end{defn}

\begin{thm}\label{SQ}
Let $C$ be a semidualizing $R$-module and let $M$ be an $R$-module.
\begin{itemize}
           \item[(i)]{$C$-$\emph{\id}_R(M) = \emph{\id}_R (C \otimes_R M) $ and $\emph{\id}_R(M) =C$-$\emph{\id}_R(\emph{\Hom}_R(C,M))$}.
           \item[(ii)]{$C$-$\emph{\pd}_R(M) = \emph{\pd}_R (\emph{\Hom}_R(C,M))$ and $\emph{\pd}_R(M) =C$-$\emph{\pd}_R(C \otimes_R M) $}.
                 \end{itemize}
\end{thm}
\begin{proof}
See \cite[Theorem 2.11]{TW}.
\end{proof}
\begin{lem}
	Let $C$ be a semidualizing $R$-module and let $M$ be a finitely generated $R$-module. Let $E$ be an injective cogenerator and $F$ be a faithfully flat $R$-module.
	\item[(i)]{One has $ C\emph{-\pd}_R(M) = C\emph{-\id}_R(\emph{\Hom}_R(M,E)) $.}
	\item[(ii)]{One has $ C\emph{-\id}_R(M) = C\emph{-\id}_R(M \otimes_R F) $.}
\end{lem}
\begin{proof}
(i). We have the following equalities
 \[\begin{array}{rl}
 C\emph{-}\id_R(\Hom_R(M,E)) &= \id_R(C \otimes_R \Hom_R(M,E)) \\
 &= \id_R(\Hom_R(\Hom_R(C , M) , E))\\
 &=\fd_R((\Hom_R(C , M))\\
 &=\pd_R(\Hom_R(C , M))\\
 &=C\emph{-}\pd_R(M),\\
  \end{array}\]
in which the first equality is from Theorem 2.7(i), the second equality is from \cite[Theorem 3.2.11]{EJ1}, the fourth equality holds since
$ \Hom_R(C , M) $ is finitely generated, and the last one is from Theorem 2.7(ii).

(ii). One has
 \[\begin{array}{rl}
 C\emph{-}\id_R(M) &= \id_R(C \otimes_R M) \\
 &= \id_R \big((C \otimes_R M) \otimes_R F\big)\\
 &= \id_R \big(C \otimes_R (M \otimes_R F)\big)\\
 &=C\emph{-}\id_R(M \otimes_R F),\\
 \end{array}\]
 in which the first and the last equalities are from Theorem 2.7(i).
\end{proof}

\section{main results}

Throughout this section, $C$ is a semidualizing $R$-module. In the rest of the paper, we use the notations $ \mathcal{P}\emph{-}\grd_R(M) $ and
$ \mathcal{I}\emph{-}\grd_R(M) $ instead of $ \grd_R(M) $ and $ E\emph{-}\dim_R(M) $, respectively.
\begin{defn}
\emph{Let $M$ be an $R$-module. The $ \mathcal{P}_C $-\textit{grade} of $M$, denoted by $ \mathcal{P}_C $-grade$ _R(M)$, is defined to be
	\begin{itemize}
		\item[]{$ \mathcal{P}_C $-grade$ _R(M) = \inf\{ i \geq 0 \mid \Ext_R^i(M,C \otimes_R P) \neq 0 $ for some projective $P \}$.}
	\end{itemize}
Also, the $ \mathcal{I}_C $-\textit{grade} of $M$, denoted by $ \mathcal{I}_C $-grade$ _R(M)$, is defined to be
	\begin{itemize}
		\item[]{$ \mathcal{I}_C $-grade$ _R(M) = \inf\{ i \geq 0 \mid \Ext_R^i(\Hom_R(C,E),M) \neq 0 $ for some injective $E \}$.}
	\end{itemize}}
\end{defn}
\begin{rem}
	\emph{One should note that if $ C = R $, then the above definition recovers the notions of \cite{R} and \cite{J1}, respectively. It is clear that
		$ \mathcal{P}_C$-$\grd_R(M) \leq \pd_R(M) $ and $ \mathcal{I}_C$-$\grd_R(M) \leq \id_R(M) $. Also, note that if $M$ is finitely generated, then \\
		\centerline{ $ \mathcal{P}_C$-$\grd_R(M) = \inf\{ i \geq 0 \mid \Ext_R^i(M,C) \neq 0 \}$.}
		For if $ \mathcal{P}_C$-$\grd_R(M) = n $ and $ \Ext_R^n(M,C \otimes_R P) \neq 0 $ for some projective $R$-module $P$, then
		$ \Ext_R^n(M,C) \otimes_R P \cong \Ext_R^n(M,C \otimes_R P) \neq 0 $ by \cite[Theorem 3.2.15]{EJ1}, and so $ \Ext_R^n(M,C) \neq 0 $. Therefore, we have
	 \[\begin{array}{rl}
	 \mathcal{P}_C\emph{-}\grd_R(M) &= \inf\{ i \geq 0 \mid \Ext_R^i(M,C) \neq 0 \}\\
	 &= \inf\{ \depth_{R_{\fp}}(C_{\fp}) \mid \fp \in \Supp_R(M) \}\\
	 &= \inf\{ \depth(R_{\fp}) \mid \fp \in \Supp_R(M) \}\\
	 &=  \mathcal{P}\emph{-}\grd_R(M).\\
	 \end{array}\]}
\end{rem}
	
\begin{lem}
Assume that $M$ is an $R$-module. The following statements hold true.
\begin{itemize}
	\item[(i)]{One has $ \mathcal{P}_C \emph{-grade}_R(M) \leq C\emph{-\pd}_R(M)$.}
	\item[(ii)]{One has $ \mathcal{I}_C \emph{-grade}_R(M) \leq C\emph{-\id}_R(M)$.}
\end{itemize}
\end{lem}
\begin{proof}
(i). We may assume that $ C $-$ \pd_R(M) = n < \infty $, since otherwise we have nothing to prove. In this case, we have $ M \in \mathcal{B}_C(R) $. By Theorem 2.7(ii), we have $ \pd_R(\Hom_R(C,M)) = n $. For any projective $R$-module $P$, we have the isomorphisms
\[\begin{array}{rl}
\Ext_R^i(M , C \otimes_R P) &\cong \Ext_R^i(C \otimes_R \Hom_R(C, M) , C \otimes_R P) \\
&\cong \Ext_{\mathcal{I}_C}^i(\Hom_R(C, M) , P)\\
&\cong \Ext_R^i(\Hom_R(C, M) , P),\\
\end{array}\]
in which the first isomorphism holds since $ M \in \mathcal{B}_C(R) $, the second isomorphism is from \cite[Theorem 4.1]{TW}, and the third isomorphism is from \cite[Corollary 4.2(b)]{TW}.
Hence $ \Ext_R^i(M,\mathcal{P}_C) = 0$ for all $ i > n $, whence $ \mathcal{P}_C \emph{-}\grd_R(M) \leq n $.

(ii). Is similar to the part (i).
\end{proof}
\begin{defn}
	\emph{Let $n$ be a non-negative integer. An $R$-module $M$ is said to be $ C $-\textit{perfect} (resp. $ n $-$ C $-\textit{perfect})
		 if $ \mathcal{P}_C \emph{-}\grd_R(M) = C \emph{-}\pd_R(M) $ (resp. $ \mathcal{P}_C \emph{-}\grd_R(M) = C \emph{-}\pd_R(M) = n $). Also,  $M$ is said to be $ C $-\textit{coperfect} (resp. $ n $-$ C $-\textit{coperfect})
		 if $ \mathcal{I}_C \emph{-}\grd_R(M) = C \emph{-}\id_R(M) $ (resp. $ \mathcal{I}_C \emph{-}\grd_R(M) = C \emph{-}\id_R(M) = n $).}
\end{defn}
		
		 Note that if $ C = R $, then the above definition recovers the notions of \cite{R} and \cite{J1}, respectively.
\begin{prop}
Let $M$ be an $R$-module.
\begin{itemize}
	\item[(i)]{If $ M \in \mathcal{B}_C(R) $, then $\mathcal{P}_C \emph{-grade}_R(M) = \mathcal{P}\emph{-\grd}_R(\emph{\Hom}_R(C,M)) $.}
	\item[(ii)]{If $ M \in \mathcal{A}_C(R) $, then $\mathcal{I}_C \emph{-grade}_R(M) = \mathcal{I}\emph{-\grd}_R(C \otimes_R M) $.}
\end{itemize}	
\end{prop}
\begin{proof}
(i). For any projective $R$-module $P$, we have the isomorphisms
\[\begin{array}{rl}
\Ext_R^i(M , C \otimes_R P) &\cong \Ext_R^i(C \otimes_R \Hom_R(C, M) , C \otimes_R P) \\
&\cong \Ext_{\mathcal{I}_C}^i(\Hom_R(C, M) , P)\\
&\cong \Ext_R^i(\Hom_R(C, M) , P),\\
\end{array}\]
in which the first isomorphism holds since $ M \in \mathcal{B}_C(R) $, the second isomorphism is from \cite[Theorem 4.1]{TW}, and the third isomorphism is from \cite[Corollary 4.2(b)]{TW}. Hence the result follows.

(ii). Is similar to the part (i).
\end{proof}
\begin{lem}\label{A2}
Let $M$ be an $R$-module with $ C $\emph{-\id}$ _R(M) < \infty $ and let $n$ be a non-negative integer. The following are equivalent:
\begin{itemize}
	\item[(i)]{$ C $\emph{-\id}$ _R(M) \leq n $.}
	\item[(ii)]{$ \emph{\Ext}_R^i(\mathcal{I}_C,M) = 0 $ for all $ i > n $.}
\end{itemize}
\end{lem}
\begin{proof}
 (i)$\Longrightarrow$(ii). By Theorem 2.7(i), we have $ \id_R(C \otimes_R M) \leq n $. Also, by \cite[Theorem 2.8(b)]{TW}, $ C \otimes_R M \in \mathcal{B}_C(R) $ since
  $ M \in \mathcal{A}_C(R) $. Now, for any injective $R$-module $I$, we have the isomorphisms
   \[\begin{array}{rl}
   \Ext_R^i(\Hom_R(C,I) , M) &\cong \Ext_R^i(\Hom_R(C,I) , \Hom_R(C, C \otimes_R M) \\
   &\cong \Ext_{\mathcal{P}_C}^i(I , C \otimes_R M)\\
   &\cong \Ext_R^i(I , C \otimes_R M),\\
   \end{array}\]
   in which the first isomorphism holds because $ M \in \mathcal{A}_C(R) $, the second isomorphism is from \cite[Theorem 4.1]{TW}, and the last one is from \cite[Corollary 4.2(a)]{TW} Now, the result follows.

    (ii)$\Longrightarrow$(i). Assume, on the contrary, that $ C $-$ \id_R(M) = t > n $. By \cite[Corollary 2.4(b)]{TW}, there is an exact sequence \\
    \centerline{$ 0 \rightarrow M \rightarrow \Hom_R(C,I^0) \rightarrow \Hom_R(C,I^1) \rightarrow \cdots \rightarrow \Hom_R(C,I^t) \rightarrow 0 $,}
    in which $ I^i $  is an injective $R$-module. Let $K^i$ be the image of the $ i $-th boundary map of the above coresolution. Note that
    $ \Ext^j_R(\mathcal{I}_C , \mathcal{I}_C) = 0 $ for all $ j \geq 1 $, by \cite[Theorem 3.2.1]{EJ1} since
    $ \Tor_j^R(C,\mathcal{I}_C) = 0 $ for all $ j \geq 1 $. Therefore a dimension shifting argument yields \\
    \centerline{$ \Ext^1_R(\mathcal{I}_C,K^{t-1}) \cong \Ext^t_R(\mathcal{I}_C , M) = 0 $.}
    Thus, in particular, the exact sequence $0 \rightarrow K^{t-1} \rightarrow \Hom_R(C,I^{t-1}) \rightarrow \Hom_R(C,I^t) \rightarrow 0 $ splits. Consequently, $ K^{t-1} $ is $ C $-injective and then  $ C $-$ \id_R(M) = t-1$, a contradiction.
\end{proof}


\begin{lem}\label{A2}
Let $M$ be an $R$-module with $ C $\emph{-\pd}$ _R(M) < \infty $ and let $n$ be a non-negative integer. The following are equivalent:
\begin{itemize}
	\item[(i)]{$ C $\emph{-\pd}$ _R(M) \leq n $}
	\item[(ii)]{$ \emph{\Ext}_R^i(M,\mathcal{P}_C) = 0 $ for all $ i > n $.}
\end{itemize}	
\end{lem}
\begin{proof}
Is dual to the proof of Lemma 3.6.
\end{proof}


\begin{prop}\label{A2}
Let $M$ be an $R$-module and let $ n $ be a non-negative integer. The following statements hold true.
	\item[(i)]{Assume that $ C $\emph{-\pd}$ _R(M) < \infty $. Then $M$ is $ n $\emph{-}$C$\emph{-}perfect if and only
		if $ \emph{\Ext}_R^i(M,\mathcal{P}_C) = 0 $ for all $ i \neq n $ and that $ \emph{\Ext}_R^n(M,C \otimes_R P) \neq 0 $ for some projective $ R $-module $P$.}
	\item[(ii)]{Assume that $ C $\emph{-\id}$ _R(M) < \infty $. Then $M$ is $ n $\emph{-}$C$\emph{-}coperfect if and only
		if $ \emph{\Ext}_R^i(\mathcal{I}_C,M) = 0 $ for all $ i \neq n $ and that $ \emph{\Ext}_R^n(\emph{\Hom}_R(C,I),M) \neq 0 $ for some injective $ R $-module $I$.}
\end{prop}
\begin{proof}
(i). Assume that $M$ is $ n $-$ C $-perfect. Then we have $ \mathcal{P}_C $-grade$ _R(M) = C $-$ \pd(M) = n$, by definition. Thus
$ \Ext_R^i(M,\mathcal{P}_C) = 0 $ for all $ i < n $. On the other hand, by Lemma 3.7, we have $ \Ext_R^i(M,\mathcal{P}_C) = 0 $ for all  $ i > n $. The remaining part follows from definition. For the converse, note that by assumption, we have $  \mathcal{P}_C $-grade$ _R(M) = n $. Now, since $ M \in  \mathcal{B}_C(R) $, in view of Lemma 3.3(i) and Lemma 3.7, we have \\
\centerline{$ n = \mathcal{P}_C $-grade$ _R(M) \leq C\emph{-}\pd_R(M) \leq n $,}
whence $M$ is $ n $-$ C $-perfect.

 (ii). Is dual of part (i).
\end{proof}

\begin{thm}\label{SQ}
Let $M$ and $N$ be $R$-modules. The following statements hold true.
\begin{itemize}
           \item[(i)]{Assume that $ \mathcal{I}_C $\emph{-grade}$ _R(M) = n $ and $ C $\emph{-\id}$ _R(N) = m < n$. Then
           	 $ \emph{\Ext}_R^i(N,M) = 0 $ for all $ i < n-m $.}
           \item[(ii)] {Assume that $ \mathcal{P}_C $\emph{-grade}$ _R(N) = n $ and $ C $\emph{-\pd}$ _R(M) = m < n$. Then
           	$ \emph{\Ext}_R^i(N,M) = 0 $ for all $ i < n-m $.}
                 \end{itemize}
\end{thm}

\begin{proof}
(i). Induct on $C$-$ \id_R(N) $. The case $C$-$ \id_R(N) = 0 $ is just definition. Assume, inductively, that $C$-$ \id_R(N) = m $. Then,
by \cite[Corollary 2.4(b)]{TW}, there exists an exact sequence \\
\centerline{$ 0 \rightarrow N \rightarrow L \rightarrow K \rightarrow 0 $,}
such that $ L \in \mathcal{I}_C $ and $ K = \coker(N \rightarrow L) $. since both $ N$ and  $L $ are in $\mathcal{A}_C(R) $, we have $ K \in \mathcal{A}_C(R) $, and therefore
$ \Tor^R_1(C,K) = 0 $. On the other hand $ C \otimes_R L \in \mathcal{I}$, by \cite[Theorem 3.2.11]{EJ1}. Hence application of $ C \otimes_R - $ on above exact sequence yields an exact sequence \\
\centerline{$ 0 \rightarrow C \otimes_R N \rightarrow C \otimes_R L \rightarrow C \otimes_R K \rightarrow 0$.}
By Theorem 2.7(i), we have $ \id_R(C \otimes_R N) = m $. Therefore $ \id_R(C \otimes_R K) = m-1 $, whence $ C $-$ \id_R(K) = m-1 $. Now, induction hypothesis applied to $L$ and $K$ yields $ \Ext^{i<n}_R(L,M) = 0 $ and $ \Ext^{i<m-(n-1)}_R(K,M) = 0 $, respectively. Therefore, the long exact sequence \\
\centerline{$ \cdots \rightarrow \Ext^i_R(L,M) \rightarrow \Ext^i_R(N,M) \rightarrow \Ext^{i+1}_R(K,M) \rightarrow \cdots$}
completes the inductive step.

(ii). Is similar to (i).
\end{proof}
\begin{cor}
Let $\fa$ and $ \fb $ be two ideals of $R$. The following statements hold true.
\begin{itemize}
\item[(i)]{If $ \mathcal{I}_C $\emph{-grade}$ _R(R / \fa) > C $\emph{-\id}$ _R(R / \fb) $, then $ \fb \colon \fa = \fb $.}
\item[(ii)]{If $ \mathcal{P}_C $\emph{-grade}$ _R(R / \fa) > C $\emph{-\pd}$ _R(R / \fb) $, then $ \fb \colon \fa = \fb $.}
\item[(iii)]{One has $ \mathcal{I}_C $\emph{-grade}$ _R(R / \fa) \leq C $\emph{-\id}$ _R(R / \fp) $, for all $ \fp \in \emph{\ass}_R(\fa) $.}
\item[(iv)]{One has $ \mathcal{P}_C $\emph{-grade}$ _R(R / \fa) \leq C $\emph{-\pd}$ _R(R / \fp) $, for all $ \fp \in \emph{\ass}_R(\fa) $.}
\end{itemize}
\begin{proof}
For (i) and (iii) use Theorem 3.9(i), and for (ii) and (iv) use Theorem 3.9(ii).
\end{proof}
\end{cor}
\begin{thm}
	Let $M$ be a finitely generated $R$-module and let $I$ be an injective $R$-module for which $ \emph{\Hom}_R(M,I) \neq 0 $. Then \\
	\centerline{$ \mathcal{P}_C $\emph{-grade}$ _R(M) \leq \mathcal{I}_C $\emph{-grade}$ _R(\emph{\Hom}_R(M,I)) \leq C $\emph{-\pd}$ _R(M) $.}
\end{thm}
\begin{proof}
	By Remark 3.2, we have $ \mathcal{P}_C $-grade$ _R(M) = \inf\{ i \geq 0 \mid \Ext_R^i(M,C) \neq 0 \}$. Now, if $ \Ext_R^i(M,C) = 0 $, then for any injective $R$-module $E$, we have the isomorphism
	$ \Tor^R_i(\Hom_R(C,E),M) \cong \Hom_R(\Ext_R^i(M,C) ,E) = 0 $, by \cite[Theorem 3.2.13]{EJ1}. Hence, we have \\
	\centerline{$ \Ext_R^i(\Hom_R(C,E),\Hom_R(M,I)) \cong \Hom_R(\Tor^R_i(\Hom_R(C,E),M),I) = 0$.}
	Therefore the inequality $ \mathcal{P}_C $-grade$ _R(M) \leq \mathcal{I}_C $-grade$ _R(\Hom_R(M,I)) $ follows. Next,
	 if $ C $-$\pd _R(M) = \infty$, then we are done. So, we can assume that $ C $-$\pd _R(M) < \infty$. Thus $ M \in \mathcal{B}_C(R) $, and then $ \Hom_R(M,I) \in \mathcal{A}_C(R) $ by \cite[Proposition 7.2(b)]{HW}. Therefore, we have the following (in)equalities
	 \[\begin{array}{rl}
	 \mathcal{I}_C\emph{-}\grd_R(\Hom_R(M,I)) &\leq C\emph{-}\id(\Hom_R(M,I))\\
	 &= \id_R(C \otimes_R \Hom_R(M,I))\\
	 &= \id_R(\Hom_R(\Hom_R(C,M),I))\\
	 &\leq \fd(\Hom_R(C,M))\\
	 &= \pd(\Hom_R(C,M)),\\
	 &= C\emph{-}\pd(M),\\
	 \end{array}\]
in which the first inequality is from Lemma 3.3(ii), the second equality is from \cite[Theorem 3.2.11]{EJ1}, the third equality holds because $ \Hom_R(C,M) $ is finitely generated, and the other equalities are from Theorem 2.7. This completes the proof.
\end{proof}
\begin{cor}
	Let $M$ be a finitely generated $R$-module and let $E$ be an injective cogenerator. Then \\
\centerline{$ \mathcal{P}_C $\emph{-grade}$ _R(M) = \mathcal{I}_C $\emph{-grade}$ _R(\emph{\Hom}_R(M,E))$.}	
\end{cor}
\begin{proof}
in view of Theorem 3.11, we need only to prove the inequality '$ \geq $'. Assume that $ i $ is a non-negative integer for which $ \Ext_R^i(\Hom_R(C,I),\Hom_R(M,E)) = 0 $ for any injective $R$-module $I$. Thus, in particular, $ \Ext_R^i(\Hom_R(C,E),\Hom_R(M,E)) = 0 $. Set $ (-)^{\vee} = \Hom_R(-,E)$. Then we have
	\[\begin{array}{rl}
	\Ext_R^i(\Hom_R(C,E),\Hom_R(M,E)) &= \Ext_R^i(C^{\vee},M^{\vee})\\
	&\cong \Tor^R_i(C^{\vee} , M)^{\vee} \\
	&\cong \Ext_R^i(M , C)^{\vee \vee}, \\
	\end{array}\]
	in which the first isomorphism is from \cite[Theorem 3.2.1]{EJ1}, and the second isomorphism is from \cite[Theorem 3.2.13]{EJ1}. Now, since
	$ \Ext_R^i(M , C) \hookrightarrow \Ext_R^i(M , C)^{\vee \vee} $, the result follows.
\end{proof}
\begin{cor}
Let $ (R, \fm) $ be local with $ \emph{\dim}(R) = n $. The following are equivalent:
	\begin{itemize}
		\item[(i)]{$R$ is Cohen-Macaulay.}
		\item[(ii)]{One has $ \mathcal{P}_C\emph{-\grd}_R(R/ \fm) = n $.}
		\item[(iii)]{One has $ \mathcal{I}_C\emph{-\grd}_R(R/ \fm) = n$.}
	\end{itemize}
\end{cor}
\begin{proof}
 (i)$\Longrightarrow$(ii). In this case, by Proposition 2.4(iv), we have $ \depth_R(C) = n $. Hence, according to Remark 3.2, we have \\
 \centerline{$ \mathcal{P}_C $-$\grd _R(R/ \fm) = \inf\{ i \geq 0 \mid \Ext_R^i(R/ \fm,C) \neq 0 \} = n $.}

 (ii)$\Longrightarrow$(iii). Use Corollary 3.12, the isomorphism $ \Hom_R(R/\fm , E(R/\fm)) \cong R/\fm $, and the fact that $ E(R/\fm) $ is an injective cogenerator.

 (iii)$\Longrightarrow$(i). Set  $ (-)^{\vee} = \Hom_R(-,E(R/\fm))$. By assumption, $ \Ext_R^i(C^{\vee},R/\fm) = 0 $ for
  all $ 0 \leq i < n$. There are isomorphisms
  	\[\begin{array}{rl}
  	\Ext_R^i(C^{\vee},R/\fm) &\cong \Ext_R^i(C^{\vee},(R/\fm)^{\vee})\\
  	&\cong \Tor^R_i(C^{\vee} , R/\fm)^{\vee} \\
  	&\cong \Ext_R^i(R/\fm , C)^{\vee \vee}, \\
  	\end{array}\]
  	where the second isomorphism is from \cite[Theorem 3.2.1]{EJ1}, and the third isomorphism is from \cite[Theorem 3.2.13]{EJ1}. Hence
  	 $ \Ext_R^i(R/\fm , C) = 0 $ for all $ 0 \leq i < n$, from which we conclude that \\
 \centerline{$ n = \mathcal{I}_C $-$\grd _R(R/ \fm) \leq \depth_R(C) = \depth(R) \leq n $.}
 Therefore $R$ is Cohen-Macaulay.
\end{proof}
\begin{thm}
	Let $E$ be an injective cogenerator. The following are equivalent:
	\begin{itemize}
		\item[(i)]{$C$ is pointwise dualizing.}
		\item[(ii)]{One has $ \mathcal{P}_C\emph{-\grd}_R(\emph{\Ext}^i_R(M,C)) \geq i$ for all finitely generated $R$-modules $M$ and all $ i \geq 0 $.}
		\item[(iii)]{One has $ \mathcal{I}_C\emph{-\grd}_R(\emph{\Tor}^i_R(M,\emph{\Hom}_R(C,E))) \geq i$ for all finitely generated $R$-modules $M$ and all $ i \geq 0 $.}
	\end{itemize}	
\end{thm}
\begin{proof}
	(i)$\Longrightarrow$(ii). In this case, $R$ is Cohen-Macaulay. For any $ \fp \in \Spec(R) $ with $ \h(\fp) < i $, we have $ \Ext^i_R(M,C)_{\fp} \cong \Ext^i_{R_{\fp}}(M_{\fp},C_{\fp}) = 0 $, since
	$ \id_{R_{\fp}}(C_{\fp}) = \h(\fp) $. Hence, if we set $ \fa = \Ann_R(\Ext^i_R(M,C)) $, then we have \\
	\centerline{$ \grd(\fa , C) =  \h_C(\fa) = \h(\fa) \geq i$.}
	Now, since $\Ext^i_R(M,C)$ is finitely generated, we have
	\[\begin{array}{rl}
	\mathcal{P}_C\emph{-}\grd_R(\Ext^i_R(M,C)) &= \inf \{j \geq 0 \mid \Ext^j_R(\Ext^i_R(M,C),C) \neq 0\}\\
	&= \inf \{j \geq 0 \mid \Ext^j_R(R/ \fa,C) \neq 0\}\\
	&\geq i.\\
	\end{array}\]
	(ii)$\Longleftrightarrow$(iii). By using Corollary 3.12 and the isomorphism \cite[Theorem 3.2.13]{EJ1}, we have
	\[\begin{array}{rl}
	\mathcal{P}_C\emph{-}\grd_R(\Ext^i_R(M,C)) &= \mathcal{I}_C\emph{-}\grd_R\big(\Hom_R(\Ext^i_R(M,C),E \big)\\
	&=  \mathcal{I}_C\emph{-}\grd_R\big(\Tor_i^R(M,\Hom_R(C,E)\big).\\
	\end{array}\]

	(ii)$\Longrightarrow$(i). Assume that $\fm$ is a maximal ideal of $R$. For any finitely generated $R$-module $M$ and any $ i \neq 0 $, we set
	$ \fa_M^i= \Ann_R(\Ext^i_R(M,C))$. Note that $ \grd_{R_{\fm}}(\fa_M^iR_{\fm} , C_{\fm}) < \infty $ since $R_{\fm}$ is Noetherian. On the other hand, by assumption, we have \\
	\centerline{$ \inf \{j \geq 0 \mid \Ext^j_R(R/ \fa_M^i ,C) \neq 0\} \geq i $,}
	for any finitely generated $R$-module $M$ and for any $ i \geq 0 $. Hence we have  \\
	\centerline{$ \inf \{j \geq 0 \mid \Ext^j_{R_{\fm}}(R_{\fm}/ \fa_M^iR_{\fm} ,C_{\fm}) \neq 0\} \geq i $,}
	for any finitely generated $R$-module $M$ and for any $ i \geq 0 $. It follows that $ \fa_M^iR_{\fm} = R_{\fm} $ for all $ i > \h(\fm) $, whence $ \Ext^i_{R_{\fm}}(M_{\fm},C_{\fm}) \cong \Ext^i_R(M,C)_{\fm} = 0 $ for all  $ i > \h(\fm) $. Therefore, since $M$ was arbitrary, we get $ \id_{R_{\fm}}(C_{\fm}) < \infty $, as wanted.
\end{proof}
Assume that $ (R,\fm) $ is local. Recall that the \textit{depth} of a
(not necessarily finitely generated) $ R $-module $ M $ is \\
\centerline{$ \depth_R(M) = \inf\{ i \geq 0 \mid \Ext_R^i(R/\fm,M) \neq 0 \} $.}
Also the \textit{width} of $M$ is defined to be \\
\centerline{$ \width_R(M) = \inf\{ i \geq 0 \mid \Tor^R_i(R/\fm,M) \neq 0 \} $.}
\begin{thm}
	Let $ (R,\fm) $ be local and $M$ be a non-zero $R$-module.
	\begin{itemize}
		\item[(i)]{One has $ \emph{\depth}(R) \leq \emph{\depth}_R(M) + C\emph{-\pd}_R(M) $.}
		\item[(ii)]{One has $ \emph{\depth}(R) \leq \emph{\width}_R(M) + C\emph{-\id}_R(M) $.}
	\end{itemize}
\end{thm}
\begin{proof}
	(i). We can assume that $ C\emph{-}\pd_R(M) < \infty $, since otherwise we have nothing to prove. We have $ \mathcal{P}_C\emph{-}\grd_R(R/ \fm) =  \mathcal{P}\emph{-}\grd_R(R/ \fm) = \depth(R)$ by Remark 3.2. Hence
	we have $ \Ext_R^i(R/\fm,M) = 0 $ for all $ i < \depth(R) - C\emph{-}\pd_R(M) $ by Theorem 3.9(ii). It follows that
	$ \depth(R) \leq \depth_R(M) + C\emph{-}\pd_R(M) $.
	
	(ii). We can assume that $ C\emph{-}\id_R(M) < \infty $, since otherwise we have nothing to prove. We have $ \mathcal{I}_C\emph{-}\grd_R(R/ \fm) =  \mathcal{P}_C\emph{-}\grd_R(R/ \fm) = \depth(R)$ by Corollary 3.12. Hence
	we have $ \Ext_R^i(M, R/\fm) = 0 $ for all $ i < \depth(R) - C\emph{-}\id_R(M) $ by Theorem 3.9(i). But there are isomorphisms
	\[\begin{array}{rl}
	\Ext_R^i(M,R/\fm) &\cong \Ext_R^i(M,\Hom_R(R/\fm,E(R/\fm)))\\
	&\cong \Hom_R(\Tor^R_i(M , R/\fm), E(R/\fm)), \\
	\end{array}\]
	in which the second equality is from \cite[Theorem 3.2.1]{EJ1}. Hence $ \Tor^R_i(M , R/\fm) = 0 $ for all $ i < \depth(R) - C\emph{-}\id_R(M) $, whence $ \depth(R) \leq \width_R(M) + C\emph{-}\id_R(M) $.
\end{proof}
\begin{lem}
	Let $M$ be a finitely generated $R$-module and let $I$ be an injective $R$-module for which $ \emph{\Hom}_R(M,I) \neq 0 $.
	\begin{itemize}
		\item[(i)]{If $M$ is $ n $-$ C $-perfect, then $ \emph{\Hom}_R(M,I) $ is $ n $-$ C $-coperfect.}
		\item[(ii)]{If $I$ is an injective cogenerator, then the converse of \emph{(i)} holds true.}
	\end{itemize}	
\end{lem}
\begin{proof}
(i). Assume that  $M$ is $ n $-$ C $-perfect. Then $ C $-$ \pd_R(M) = n $, and hence $ M \in \mathcal{B}_C(R) $. Therefore, by \cite[Proposition 7.2(b)]{HW},  we have $ \Hom_R(M,I) \in \mathcal{A}_C(R) $. Hence we have the (in)equalities
\[\begin{array}{rl}
\mathcal{P}_C\emph{-}\grd_R(M) &\leq \mathcal{I}_C\emph{-}\grd_R(\Hom_R(M,I))\\
&\leq C\emph{-}\id_R(\Hom_R(M,I))\\
&= \id_R(C \otimes_R \Hom_R(M,I))\\
&= \id_R(\Hom_R(\Hom_R(C,M),I))\\
&\leq \fd(\Hom_R(C,M))\\
&= \pd(\Hom_R(C,M))\\
&= C\emph{-}\pd(M),\\
\end{array}\]
in which the first inequality is from Theorem 3.11, the second inequality is from Lemma 3.3(ii), and the first and the last equalities are from Theorem 2.7. Hence $ \Hom_R(M,I) $ is $ n $-$ C $-coperfect.

(ii). Use Lemma 2.8(i) and Corollary 3.12.
\end{proof}
\begin{cor}
	Let $ (R, \fm) $ be local with $ \emph{\dim}(R) = n $. The following are equivalent:
	\begin{itemize}
		\item[(i)]{$R$ is regular.}
		\item[(ii)]{$ R / \fm $ is n-$ C $-perfect.}
		\item[(iii)]{$ R / \fm $ is n-$ C $-coperfect.}
	\end{itemize}
\end{cor}
\begin{proof}
	(i)$\Longrightarrow$(ii). In this case, by \cite[Corollary 8.6]{C}, we have $ C = R $, and then \\
	\centerline{$ \mathcal{P}$-$\grd _R(R/ \fm) = \inf\{ i \geq 0 \mid \Ext_R^i(R/ \fm,R) \neq 0 \} = n = \pd_R(R/ \fm) $.}
	
	(ii)$\Longrightarrow$(iii). Use Lemma 3.16, the isomorphism $ \Hom_R(R/\fm , E(R/\fm)) \cong R/\fm $, and the fact that $ E(R/\fm) $ is an injective cogenerator.
	
	(iii)$\Longrightarrow$(i). Note that $ C \otimes_R R/\fm $ is a finite dimensional vector space over $ R/\fm $. Therefore, by Theorem 2.7(i) and the assumption, we have \\
	\centerline{$ \id_R(R/\fm) = \id_R(C \otimes_R R/\fm) = C\emph{-}\id_R(R/\fm) = n $.}
	Hence $R$ is regular by \cite[Exercise 3.1.26]{He}.
\end{proof}
\begin{thm}
	Let $ (R, \fm) $ be a Cohen-Macaulay local ring. Let $M$ be a finitely generated $R$-module with $ C\emph{-\pd}_R(M) < \infty $ and let $ n = \emph{\depth}(R) - \emph{\depth}(M) $. The following are equivalent:
	\begin{itemize}
		\item[(i)]{$M$ is Cohen-Macaulay.}
		\item[(ii)]{$M$ is  $ n $-$ C $-perfect.}
		\item[(iii)]{$\emph{\Hom}_R(M,E(R/ \fm)) $ is  $ n $-$ C $-coperfect.}
	\end{itemize}
\end{thm}
\begin{proof}
(i)$\Longrightarrow$(ii). Since $ C\emph{-}\pd_R(M) < \infty $, we have $ M \in \mathcal{B}_C(R) $ and hence $ \Ext^i_R(C,M) = 0 $ for all $ i \geq 0 $. Thus $ \depth_R(\Hom_R(C,M)) = \depth_R(M) $ by \cite[Corollary 5.2.7]{CF}. Therefore we have the equalities
  \[\begin{array}{rl}
  C\emph{-}\pd_R(M) &= \pd_R(\Hom_R(C,M)\\
  &= \depth(R) - \depth_R(\Hom_R(C,M)) \\
  &=  \depth(R) - \depth_R(M)\\
  &=  n,\\
  \end{array}\]
  in which the first equality is from Theorem 2.7(ii), and the second equality is Auslander-Buchsbaum formula. Hence, in view of Lemma 3.7, we have $ \Ext^i_R(M,\mathcal{P}_C) = 0 $ for all $ i > n $. Next, since $R$ is Cohen-Macaulay, there are equalities
    \[\begin{array}{rl}
    \inf\{ i \geq 0 \mid \Ext_R^i(M,C) \neq 0 \} &= \inf\{ i \geq 0 \mid \Ext_R^i(R/\Ann_R(M),C) \neq 0 \}\\
    &= \grd(\Ann_R(M) , C) \\
    &= \grd(\Ann_R(M) , R)  \\
    &=  \dim(R) - \dim_R(M)\\
    &=  n.\\
    \end{array}\]
Thus, since $M$ is finitely generated, we have $ \Ext^i_R(M,\mathcal{P}_C) = 0 $ for all $ i < n $ by \cite[Theorem 3.2.15]{EJ1}. Hence $M$ is $ n $-$ C $-perfect by Proposition 3.8(i).

 (ii)$\Longrightarrow$(iii). Use Lemma 3.16 and the fact that $ E(R/\fm) $ is an injective cogenerator.

 (iii)$\Longrightarrow$(i). We have the equalities
     \[\begin{array}{rl}
     \dim_R(M) &= \dim(R) - \mathcal{P}\emph{-}\grd_R(M)\\
     &= \dim(R) - \mathcal{P}_C\emph{-}\grd_R(M)\\
     &= \dim(R) - \mathcal{I}_C\emph{-}\grd_R\big(\Hom_R(M,E(R/ \fm))\big) \\
     &= \dim(R) - n\\
     &= \depth(M), \\
     \end{array}\]
in which the first equality holds since $R$ is Cohen-Macaulay, the second equality holds by Remark 3.3, the third equality holds by Corollary 3.12, and the fourth equality holds by assumption. Hence, $M$ is Cohen-Macaulay.
\end{proof}
\begin{thm}
	Let $ (R, \fm) $ be local with $ \emph{\dim}(R) = n $. The following are equivalent:
	\begin{itemize}
		\item[(i)]{$C$ is dualizing.}
		\item[(ii)]{$E(R/ \fm) $ is  n-$ C $-perfect.}
		\item[(iii)]{$\emph{\Hom}_R(C,E(R/ \fm)) $ is  n-perfect.}
		\item[(iv)]{$\widehat{R}$ is n-$ C $-coperfect.}
		\item[(v)]{$\widehat{C}$ is n-coperfect.}
	\end{itemize}
\end{thm}
\begin{proof}
(i)$\Longrightarrow$(ii). First, note that $ \fd_R(\Hom_R(C,E(R/ \fm))) = n $. Hence by \cite[Corollary 3.4]{F1}, we have $ \pd_R(\Hom_R(C,E(R/ \fm))) \leq n $, from which we conclude that $ C $-$ \pd_R(E(R/ \fm)) = \pd_R(\Hom_R(C,E(R/ \fm))) = n $ by Theorem 2.7(ii). Next, note that the minimal injective resolution of
$ C $ is of the form \\
\centerline{$ 0 \rightarrow C \rightarrow \underset{\h(\fp) = 0} \bigoplus E(R/ \fp) \rightarrow \underset{\h(\fp) = 1} \bigoplus E(R/ \fp) \rightarrow \cdots \rightarrow E(R/ \fm) \rightarrow 0 $. $ (\dagger) $}
Let $P$ be a projective $R$-module. Applying the exact functor $ - \otimes_R P $ on the $ (\dagger) $, we get the following exact complex \\
\centerline{$ 0 \rightarrow C \otimes_R P \rightarrow \underset{\h(\fp) = 0} \bigoplus E(R/ \fp) \otimes_R P \rightarrow \underset{\h(\fp) = 1} \bigoplus E(R/ \fp) \otimes_R P \rightarrow \cdots \rightarrow E(R/ \fm) \otimes_R P \rightarrow 0 $.}
which is an injective resolution for  $ C \otimes_R P $. By \cite[Theorem 3.3.12]{EJ1}, the injective $R$-module $ E(R / \fp) \otimes_R P$ is a direct sum of copies of
$ E(R / \fp)$ for each $ \fp \in \Spec(R)$. Therefore, since $ \Hom_R(E(R/ \fm),E(R/ \fp)) = 0 $ for any prime ideal $ \fp \neq \fm $, we have
$ \Ext_R^i(E(R/ \fm), \mathcal{P}_C) = 0 $ for all $  i \neq n $. But, if $ P = R $, then
  \[\begin{array}{rl}
  \Ext_R^n(E(R/ \fm), C) &= \Hom_R(E(R/ \fm), E(R/ \fm))\\
  &\cong \widehat{R} \\
  &\neq 0.\\
  \end{array}\]
Hence,  $ E(R/ \fm) $ is $ n $-$ C $-perfect by proposition 3.8(i), as wanted.

(ii)$\Longleftrightarrow$(iii). Note that $ \Hom_R(C,E(R/ \fm)) \in \mathcal{A}_C(R) $. Now,
$E(R/ \fm) $ is  $ n $-$ C $-perfect if and only if $ \mathcal{P}_C \emph{-}\grd_R(E(R/ \fm)) = C\emph{-}\pd_R(E(R/ \fm)) = n$, and this is the case if and only if $ \mathcal{P} \emph{-}\grd_R(\Hom_R(C,E(R/ \fm))) = \pd_R(\Hom_R(C,E(R/ \fm))) = n $ by Theorem 2.7(ii) and Lemma 3.3(i).

(i)$\Longrightarrow$(iv). First, note that there are equalities \\
\centerline{$ n = \id_R(C) = C\emph{-}\id_R(R) = C\emph{-}\id_R(\widehat{R}) $}
in which the second equality is from Theorem 2.7(i) and the last one is from Lemma 2.8(ii). Hence, we have $ \Ext_R^i(\mathcal{I}_C,\widehat{R}) = 0 $ for all $ i > n $ by Lemma 3.6.  Next, we show that $ n $ is the least
 integer for which $ \Ext_R^n(\mathcal{I}_C,\widehat{R}) \neq 0 $. To do this, first we show that if $ \Ext_R^i(\mathcal{I}_C,\widehat{R}) \neq 0 $ for some integer $ i $, then $ \Ext_R^i(\Hom_R(C,E(R/\fm)),\widehat{R}) \neq 0 $. Assume that $ I $ is an injective $R$-module for which
 $ \Ext_R^i(\Hom_R(C,I),\widehat{R}) \neq 0 $. Suppose, on the contrary, that $ \fm \notin \Ass_R(I) $. Without loss of generality, we can assume that $ I = E(R/\fp) $ with $ \fp \neq \fm $. We have the isomorphisms
  \[\begin{array}{rl}
  \Ext_R^i(\Hom_R(C,I),\widehat{R}) &\cong \Ext_R^i\big(\Hom_R(C,I), \Hom_R(E(R/\fm),E(R/\fm))\big) \\
  &\cong \Hom_R^i\big(\Tor^R_i(\Hom_R(C,I),E(R/\fm)) , E(R/\fm)\big),\\
  \end{array}\]
in which the second isomorphism is from \cite[Theorem 3.2.1]{EJ1}. Hence $ \Ext_R^i(\Hom_R(C,I),\widehat{R}) \neq 0 $ if and only if $ \Tor^R_i(\Hom_R(C,I),E(R/\fm)) \neq 0 $. Choose an element $ x \in \fm \smallsetminus \fp $. Then multiplication of $ x $ induces an isomorphism on $ I $ and hence on $ \Hom_R(C,I) $. It follows that the homomorphism \\
\centerline{$ \Tor^R_i(\Hom_R(C,I),E(R/\fm)) \overset{x.}\longrightarrow \Tor^R_i(\Hom_R(C,I),E(R/\fm)) $}
is both an isomorphism and locally nilpotent, whence $ \Tor^R_i(\Hom_R(C,I),E(R/\fm)) = 0 $, a contradiction. Consequently, we need only to show that $ n $ is the least
integer for which $ \Ext_R^n(\Hom_R(C,E(R/\fm)),\widehat{R}) \neq 0 $ or equivalently $ \Tor^R_n(\Hom_R(C,E(R/\fm)),E(R/\fm)) \neq 0 $. Applying the exact functor $ \Hom_R(-,E(R/\fm)) $ on $ (\dagger) $ above, we get a flat resolution \\
\centerline{$ 0 \rightarrow F_n \rightarrow \cdots \rightarrow F_1 \rightarrow F_0 \rightarrow \Hom_R(C,E(R/ \fm)) \rightarrow 0 $,}
in which $ F_i = \Hom_R\Big(\underset{\h(\fp) = i} \bigoplus E(R/ \fp) , E(R/\fm)\Big)$. Now, by applying $ - \otimes_R E(R/\fm) $, on this flat resolution, we can compute $ \Tor^R_i(\Hom_R(C,E(R/\fm)),E(R/\fm)) $. By the above argument, $ \Tor^R_i(\Hom_R(C,E(R/\fm)),E(R/\fm)) = 0 $ for all $ 0 \leq i < n $. On the other hand, we have
  \[\begin{array}{rl}
  \Tor^R_n(\Hom_R(C,E(R/\fm)),E(R/\fm)) &= F_n \otimes_R E(R/\fm) \\
  &= \Hom_R(E(R/\fm),E(R/\fm)) \otimes_R E(R/\fm)\\
  &= \widehat{R} \otimes_R E(R/\fm) \\
  &\cong E(R/\fm)\\
   &\neq 0.\\
  \end{array}\]
Thus $ \Ext_R^i(\mathcal{I}_C,\widehat{R}) = 0 $ for all $ i \neq 0 $ and $ \Ext_R^n(\Hom_R(C,E(R/\fm)),\widehat{R}) \neq 0 $. Consequently,
$ \widehat{R} $ is $ n $-$ C $-coperfect by Proposition 3.8(ii).

(iv)$\Longleftrightarrow$(v). Note that $ \widehat{R} \in \mathcal{A}_C(R) $, and so $ \widehat{C} \cong C \otimes_R \widehat{R} \in \mathcal{B}_C(R) $ by \cite[Theorem 2.8(b)]{TW}. Now,
$\widehat{R} $ is  $ n $-$ C $-coperfect if and only if $ \mathcal{I}_C \emph{-}\grd_R(\widehat{R}) = C\emph{-}\id_R(\widehat{R}) = n$, and this is the case if and only if $ \mathcal{I} \emph{-}\grd_R(\widehat{C}) = \id_R(\widehat{C}) = n $ by Theorem 2.7(i) and Lemma 3.3(ii).

(iv)$\Longrightarrow$(i). One has \\
\centerline{$ \id_R(C) = C\emph{-}\id_R(R) = C\emph{-}\id_R(\widehat{R}) = n $}
 where the first equality is from Theorem 2.7(i), the second equality is from Lemma 2.8(ii), and the last equality is the assumption. Hence $ C $ is dualizing.
\end{proof}
Recall that for an $R$-module $M$, the $ i $-th local cohomology module of $M$ with respect to an ideal $ \fa $ of $R$, denoted by $ \H^i_{\fa}(M) $, is defined to be \\
\centerline{$  \H^i_{\fa}(M) = \underset{\underset{n \geq 1} \longrightarrow}  \lim \Ext^i_R(R/ \fa^n , M) $.}
\begin{thm}
	Let $ (R, \fm) $ be local with $ \emph{\dim}(R) = n $. The following are equivalent:
	\begin{itemize}
		\item[(i)]{$R$ is Gorenstein.}
		\item[(ii)]{any quasidualizing $R$-module is $ n $-perfect.}
		\item[(iii)]{there exists an $ n $-perfect quasidualizing $R$-module with \emph{1-}dimensional socle.}
	\end{itemize}
\end{thm}
\begin{proof}
(i)$\Longrightarrow$(ii). Assume that $ T $ is a quasidualizing $R$-module. By \cite[Proposition 2.1]{K}, $T$ is a quasidualizing
$ \widehat{R} $-module. Also, since $ \widehat{R} $ is a complete, by \cite[Theorem 3.1]{K}, there exists a semidualizing $ \widehat{R} $-module $K$ for which
$ K \cong \Hom_{\widehat{R}}(T,E(R/\fm)) $. But, as $ \widehat{R} $ is Gorenstein, we have $ K \cong \widehat{R} $ by \cite[Corollary 8.6]{C}.  Therefore we have the isomorphisms
   \[\begin{array}{rl}
   T &\cong \Hom_{\widehat{R}}(\Hom_{\widehat{R}}(T,E(R/\fm)),E(R/\fm))\\
     &\cong \Hom_{\widehat{R}}(\widehat{R},E(R/\fm))\\
     &\cong E(R/\fm),\\
   \end{array}\]
in which the first isomorphism holds because $ T $ is Matlis reflexive over $ \widehat{R} $.
Now, if $P$ is a projective $R$-module, then it is free and hence by using the minimal injective resolution of $R$, we have
$ \mu^i(\fp,P) = 0 $ for all $ i \neq \h(\fp) $. Therefore $ \Ext_R^i(E(R/\fm),P) = 0 $ for all $ i \neq n$, and that \\
\centerline{$ \Ext_R^n(E(R/\fm),P) \cong \Hom_R(E(R/\fm),E(R/\fm)^{\mu^n(\fm,P)}) \neq 0 $.}
Thus, since $ \pd_R(E(R/\fm)) = n $, the $R$-module $ T = E(R/\fm) $ is $ n $-perfect by Proposition 3.8(i).

(ii)$\Longrightarrow$(iii). The desired module is $ E(R/\fm) $.

(iii)$\Longrightarrow$(i). Assume that $ T $ is an $ n $-perfect quasidualizing $R$-module with $ 1 $-dimensional socle. Observe that, by \cite[Proposition 2.1]{K}, $T$ is a quasidualizing $\widehat{R}$-module. Also, we have $ \pd_{\widehat{R}}(T) = n $ since
 $ T \otimes_R \widehat{R} \cong T $. Next, a similar argument as in \cite[Lemma 9.1.4]{EJ1} yields the isomorphisms \\
\centerline{ $ \Ext^i_R(T,R) \cong \Ext^i_{\widehat{R}}(T,\widehat{R}) $,}
for all $ i \geq 0 $. Hence, if $ \Ext^i_R(T,F) = 0$ for any free $R$-module $F$, then $ \Ext^i_{\widehat{R}}(T,X) = 0$ for any free $\widehat{R}$-module $X$. Finally, if $ \vdim_{R/\fm}(\Soc_R(T)) = 1 $, then $ \Hom_R(R/\fm,T) \cong R/\fm $ and therefore
$ \Hom_{\widehat{R}}(R/\fm,T) \cong R/\fm $, whence $ \vdim_{R/\fm}(\Soc_{\widehat{R}}(T)) = 1 $. Consequently, $ T $ is an $ n $-perfect $ \widehat{R} $-module  with $ 1 $-dimensional socle. On the other hand, $ R $ is Gorenstein if and only
 if $ \widehat{R} $ is so. Hence, we can replace $R$ by  $\widehat{R}$ and assume that $R$ is complete. In the new case, by \cite[Theorem 3.1]{K}, there exists a semidualizing
 $ R $-module $C$ for which $ C \cong \Hom_R(T,E(R/\fm)) $. Thus we have $ \id_R(C) \leq \pd_R(T) = n $, whence $C$ is dualizing, and hence the
 (in)equalities \\
 \centerline{$ \dim(R) = \dim_R(C) \leq \id_R(C) = \depth(R) $,}
 show that $ R $ is Cohen-Macaulay. Now, we can use local duality \cite[Theorem 11.2.8]{BS} to get the isomorphisms \\
 \centerline{$ \H^n_{\fm}(R) = \Hom_R(C,E(R/\fm)) \cong T $,}
 from which we conclude that  $ \vdim_{R/\fm}(\H^n_{\fm}(R)) = 1 $.
 Consider the following Grothendieck's third quadrant spectral sequence \cite[theorem 10.47]{Ro} \\
 \centerline{$ \E_2^{p,q} = \Ext^p_R(R/\fm,\H^q_{\fm}(R)) \underset{p}\Longrightarrow \Ext^{p+q}_R(R/\fm,R) $.}
 Since $R$ is Cohen-Macaulay, we have $ \H^q_{\fm}(R) = 0 $ for all $ q \neq n $ by \cite[Corollary 6.2.9]{BS}, and hence  $ \E_2^{p,q} $ collapses on the $ n $-th column, whence \\
 \centerline{$ \Ext^n_R(R/\fm,R) \cong \Hom_R(R/\fm,\H^n_{\fm}(R)) \cong R/\fm $.}
 It follows, from \cite[Theorem 9.2.27]{EJ1}, that $ R $ is Gorenstein.
\end{proof}
By using the same argument as in the proof of the above theorem, we have the following corollary which is a characterization of Cohen-Macaulay local rings.
\begin{cor}
	Let $ (R, \fm) $ be a complete local ring with $ \emph{\dim}(R) = n $. The following are equivalent:
	\begin{itemize}
		\item[(i)]{$R$ is Cohen-Macaulay.}
		\item[(ii)]{there exists an $ n $-perfect quasidualizing $R$-module.}
	\end{itemize}
\end{cor}
\begin{proof}
(i)$\Longrightarrow$(ii). By assumption, $R$ has a dualizing module, say $D$. Now, by Theorem 3.19 and \cite[theorem 3.1]{K}, $ \Hom_R(D,E(R/\fm)) $ is an $ n $-perfect quasidualizing $R$-module.

(ii)$\Longrightarrow$(i). Suppose that $ T $ is an $ n $-perfect quasidualizing $R$-module. Then by \cite[theorem 3.1]{K}, there exists a semidualizing module $C$ for which $ T = \Hom_R(C,E(R/\fm)) $. Now, we have $ \id_R(C) \leq \pd(T) = n $, whence $C$ is dualizing and $R$ is Cohen-Macaulay.
\end{proof}
\begin{ques}
\emph{Can we omit the condition '1-dimensional socle' in the Theorem 3.20? More precisely, if there exists an $ n $-perfect quasidualizing $R$-module, then is $R$ Gorenstein?}
\end{ques}

\textbf{Acknowledgment.} The authors are grateful to the referee for his/her invaluable comments.

\bibliographystyle{amsplain}

\begin{thebibliography}{9}


\bibitem{He}
~W. Bruns and ~J. Herzog , \emph{Cohen-Macaulay rings,} Cambridge University Press, Cambridge, 1993.


\bibitem{BS}
~M.P. Brodmann and ~R.Y. Sharp , \emph{Local cohomology: an algebraic introduction with geometric applica-
tions}, Cambridge University Press, Cambridge, 1998.


\bibitem{C}
~L.W. Christensen,  \emph{Semi-dualizing complexes and their Auslander categories,} Trans. Amer. Math. Soc.
\textbf{5} (2001), 1839--1883.


\bibitem{CF}
~L.W. Christensen and ~H.-B. Foxby, \emph{Hyperhomological Algebra
	with Applications to Commutative Rings,} Available from http://www.math.ttu.edu/~lchriste/download/918-final.pdf.


\bibitem{EJ1}
~E. Enochs and ~O. Jenda, \emph{Relative Homological Algebra}, de Gruyter Expositions in
Mathematics 30, 2000.



\bibitem{F}
~H.-B. Foxby, \emph{Gorenstein modules and related modules}, Math. Scand. \textbf{31} (1973), 267--284.


\bibitem{F1}
~H.-B. Foxby, \emph{Isomorphisms between complexes with applications to the homological theory of modules}, Math. Scand. \textbf{40} (1977), 5--19.


\bibitem{G}
~E. S. Golod, \emph{$G$-dimension and generalized perfect ideals}, Trudy Mat. Inst. Steklov. Algebraic geometry and its applications \textbf{165} (1984), 62--66.


\bibitem{H1}
~R. Hartshorne, \emph{Local cohomology}, A seminar given by A. Grothendieck, Harvard University,
Fall, vol. 1961, Springer-Verlag, Berlin, 1967.



\bibitem{HJ}
~H. Holm, ~P. J\o rgensen, \emph{Semi-dualizing modules and related Gorenstein homological dimensions}, J.
Pure Appl. Algebra, \textbf{205}(2006) 423--445.


\bibitem{HW}
~H. Holm, ~D. White, \emph{Foxby equivalence over associative rings,} J. Math. Kyoto Univ. \textbf{47} no.4,
(2007), 781--808.


\bibitem{J1}
~O.M. Jenda \emph{The dual of the grade of a module,} Arch. Math. \textbf{51} (1988), 297--302.



\bibitem{K}
~B. Kubik, \emph{ Quasidualizing modules}, J. Commut. Algebra. \textbf{6} (2014), 209--229.


\bibitem{R}
~D. Rees, \emph{The grade of an ideal or module,} Proc. Cambridge Phil. Soc. \textbf{53} (1957) 28--42.



\bibitem{Ro}
~J.J. Rotman, \emph{An introduction to homological algebra,}  Second ed., Springer, New York, 2009..



\bibitem{TW}
 ~R. Takahashi and ~D.White, \emph{Homological aspects of semidualizing modules}, Math. Scand. \textbf{106} (2010)
5--10.



\bibitem{V}
~W. V. Vasconcelos, \emph{Divisor theory in module categories,} North-Holland Math. Stud. 14, North-Holland
Publishing Co., Amsterdam (1974).

\end{thebibliography}

\end{document}